\documentclass[a4paper,12pt]{article}
\usepackage{amssymb}
\usepackage{latexsym}
\usepackage{amsmath,amsthm,amssymb}
\usepackage{amsfonts}
\usepackage{eufrak}

\newtheorem{thm}{Theorem}[section]
\newtheorem{lem}[thm]{Lemma}

\theoremstyle{definition}
\newtheorem{defn}[thm]{Definition}
\newtheorem{exa}[thm]{Example}
\newtheorem{rem}[thm]{Remark}

\begin{document}

\begin{center}
{\Large On norms in some class of exponential type Orlicz spaces of random variables}
\end{center}
\begin{center}
{\sc Krzysztof Zajkowski}\\
Institute of Mathematics, University of Bialystok \\ 
Ciolkowskiego 1M, 15-245 Bialystok, Poland \\ 
kryza@math.uwb.edu.pl 
\end{center}

\begin{abstract}
A new characterization of the exponential type Orlicz spaces generated by the functions $\exp(|x|^p)-1$ ($p\ge 1$) is given. 
We define norms for centered random variables belonging to these spaces. We show equivalence of these norms with the Luxemburg norms.
On the example of Hoeffding's inequality we present some application of these norms in a probabilistic context. 

\end{abstract}



\section{Introduction}
The most important class of exponential type Orlicz spaces form spaces generated by the functions $\psi_p(x)=\exp(|x|^p)-1$ ($p\ge 1$). In probability theory, for $p=1$, we have the space of sub-exponential random variables and, for $p=2$, the space of sub-gaussian random variables. 

Spaces of this type also appear naturally in  asymptotic geometric analysis (see the monograph of Arstein-Avidan et al. \cite{A-AGM}), especially, in the context of concentration of measures (see \cite[Par. 3.5.2]{A-AGM}) and \cite{GPV}, for instance), convex geometry (see, e.g., \cite{Bour, BM, HPTh}) and also in  information-based complexity (see Hinrichs et al. \cite{HPU}). 

Let us recall that in probability theory, for any $p\ge 1$, the Orlicz space $L_{\psi_p}$ consists of all random variables $X$  on a probability space  for which the Luxemburg norm 
$$
\|X\|_{\psi_p}:=\inf \big\{K>0:\; \mathbb{E}\exp(|X/K|^p)\le 2\big\}
$$
is finite (see Luxemburg \cite{Lux}). 
Let us note that $L_{\psi_{p_1}}\subset L_{\psi_{p_2}}$, if $p_1\ge p_2$, and moreover $L_\infty\subset L_{\psi_p} \subset L_r$ ($p,r\ge 1$), where $L_\infty,\;L_r$ denote the classical Lebesgue spaces. In other words, the  spaces $L_{\psi_p}$ form an increasing family, by decreasing $p$, smaller than all of $L_p$-spaces but larger than the space of bounded random variables $L_\infty$.

Estimates of the $\psi_\alpha$-norm play often an important role in solving problems of classical probability theory and asymptotic geometric analysis.  
This can be quite challenging and delicate task, and therefore 
any alternative form of this norm may be useful.
It is known an equivalent expression of the $\psi_p$-norms in terms of the classical Lebesgue norms 
as
$\sup_{r\ge p}r^{-1/p}\|X\|_{L_r}$; see \cite[Lem. 3.5.5]{A-AGM}.  Some modification of the $\psi_\alpha$-norm one can find in Dick et al. \cite{DHPP}.  
In this paper I would like to present an equivalent form of the $\psi_p$-norms on spaces of centered random variables.

Let us recall that centered sub-gaussian random variables have another important classical characterization (see Kahane \cite{Kahane}): a random variable $X$ 
is sub-gaussian if there exists a positive constant $K$ such that $\mathbb{E}\exp(tX)\le \exp(K^2t^2/2)$ for all $t\in\mathbb{R}$. In other words when the moment generating function of $X$ is majorized by the moment generating function of a centered Gaussian variable $g$ with the variance $K^2$;
recall that $\mathbb{E}\exp(tg)=\exp(K^2t^2/2)$ for $g\sim\mathcal{N}(0,K^2)$. 
Estimates for moments generating functions play a crucial role in the theory of concentration inequalities as in the proofs of Hoeffding's \cite{Hoeff} and Azuma's \cite{Azuma} inequalities, for instance.

It is possible to introduce another  norm, equivalent to $\|\cdot\|_{\psi_2}$ in the space of centered sub-gaussian random variables, of the form
$$
\tau(X):=\inf\big\{K>0:\;\forall_{t\in\mathbb{R}}\;\mathbb{E}\exp(tX)\le \exp(K^2t^2/2)\big\};
$$
see Buldygin and Kozachenko \cite[Th. 1.1.2]{BulKoz}. In Vershynin \cite[Prop. 2.5.2]{Ver} one can find the proof
of equivalence of
the norms $\|\cdot\|_{\psi_2}$ and $\tau(\cdot)$.

In general it is a problem to define norms in the above form for any $p\ge 1$, because functions that majorize of the logarithm of the moment generating functions 
must be not only $C$-functions (see \cite[Def. 2.2.1]{BulKoz}) but must also satisfy a quadratic condition (to be quadratic functions in a neighborhood of zero; compare Buldygin and Kozachenko \cite[Def. 2.4.1]{BulKoz}).

In Definition \ref{phip} we propose a standardization of the function $|x|^p$ which is equivalent and of the form: $\varphi_p(x)=x^2/2$ if $|x|\le 1$ and  $\varphi_p(x)=|x|^p/p-1/p+1/2$ if $|x|>1$. 
Similarly as in the sub-gaussian case, one can define 
a norm of centered random variables as follows:
$$
\tau_{\varphi_p}(X):=\inf\{K>0:\;\forall_{t\in\mathbb{R}}\;\mathbb{E}\exp(tX)\le \exp\varphi_q(Kt)\}\quad(q=p/(p-1)).
$$
Let us note that for $p=2$ the norm $\tau_{\varphi_2}$ coincides with $\tau$. 

Let $L_{\psi_p}^0$ denote the space of centered random variables belonging to $L_{\psi_p}$, i.e., $L_{\psi_p}^0:=\{X\in L_{\psi_p}:\;\mathbb{E}X=0\}$.
In the main Theorem \ref{mgf} we show equivalence of the norms $\|\cdot\|_{\psi_p}$ and $\tau_{\varphi_p}(\cdot)$ on $L_{\psi_p}^0$. 



\section{Norms on  $L_{\psi_p}^0$-spaces}
Before we pass to a more detailed characterization of the norm $\tau_{\varphi_p}$, we first present a certain lemma about  the elements in the spaces 
$L_{\psi_p}$.
\begin{lem}
\label{charlem}
Let $X$ be a random variable and $p\ge 1$. There exist positive constants $K,L,M$ 
such that 
the following conditions are equivalent:\\
 1. 
 $\mathbb{E}\exp\{|X/K|^p\}\le 2$;\\
 2. 
 $\mathbb{P}(|X|\ge t) \le 2\exp\{-(t/L)^p\}$ for all $t \ge 0$;\\ 
 3. 
 $\mathbb{E}|X|^\alpha\le  2M^\alpha\Gamma\big(\frac{\alpha}{p}+1\big)$ for all $\alpha\ge 1$.
\end{lem}
\begin{proof}
$1. \Rightarrow 2.$
Using the exponential Markov inequality and then property 1 we get
$$
\mathbb{P}(|X|\ge t)=\mathbb{P}\Big(e^{|X/K|^p}\ge e^{(t/K)^p}\Big)\le e^{-(t/K)^p}\mathbb{E}e^{|X/K|^p}\le 2e^{-(t/K)^p}.
$$
It means that assuming property 1, property 2 follows with $L=K$.

$2. \Rightarrow 3.$ Using property 2 we get 
$$
 \mathbb{E}|X|^\alpha =\int_0^\infty \mathbb{P}(|X|\ge t)\alpha t^{\alpha-1}dt \le 2\alpha\int_0^\infty e^{-(t/L)^p} t^{\alpha-1}dt.
$$
After change of variables $u = (t/L)^p$ we obtain that
$$ 
2\alpha\int_0^\infty e^{-(t/L)^p} t^{\alpha-1}dt \le 2L^\alpha\frac{\alpha}{p}\int_0^\infty e^{-u} u^{\frac{\alpha}{p}-1}du=2L^\alpha
\Gamma\Big(\frac{\alpha}{p}+1\Big).
$$
Property 3 follows with $M=L$ and in consequence equals $K$.

$3. \Rightarrow 1.$ 
Using the Taylor expansion for the exponential function, Lebesgue's monotone convergence theorem and property 3 we obtain the following estimate
$$
\mathbb{E}e^{|X/K|^p} = 1 +
\sum_{\alpha=1}^\infty
\frac{\mathbb{E}|X/K|^{p\alpha}}{\alpha!} \le 1 +
2\sum_{\alpha=1}^\infty
\frac{M^{p\alpha}\Gamma(\alpha+1)}{K^{p\alpha}\alpha!}=1+2\sum_{\alpha=1}^\infty \Big(\frac{M^p}{K^p}\Big)^{\alpha},
$$
which is less or equal 2 for $K\ge 3^{1/p}M$.
\end{proof}
\begin{rem}
\label{mrem}
Condition 1 in the above lemma means that $X$ belongs to the exponential type Orlicz space $L_{\psi_p}$ and its norm satisfies the bound $\|X\|_{\psi_p}\le K$.
Moreover if  condition 2 is satisfied with some constant $L$, then $\|X\|_{\psi_p}\le 3^{1/p}L$.
\end{rem}



Let us define some family of functions on $\mathbb{R}$. 
\begin{defn}
\label{phip}
Let for $p \ge 1$
$$
\varphi_p(x)=\left\{
\begin{array}{ccl}
\frac{x^2}{2}, & {\rm if} & |x|\le 1,\\
\frac{1}{p}|x|^p-\frac{1}{p}+\frac{1}{2}, & {\rm if} & |x|>1.
\end{array}
\right.
$$
We define $\varphi_\infty(x)=x^2/2$ for $|x|\le 1$ and $\varphi_\infty(x)=\infty$ if $|x|>1$.
\end{defn} 
The functions $\varphi_p$ for $p\in [1,\infty)$ are examples of  $C$-functions 
that satisfy the quadratic condition. 

Let us recall now the notion of equivalence of functions.
Let  $f_1$ and $f_2$ be even real-valued functions. We say that $f_1$ is weaker than $f_2$ ($f_1\prec f_2$) if there exist $x_0\ge 0$ and $k>0$ such that 
$f_1(x)\le f_2(kx)$ for all $x\ge x_0$. Functions $f_1$ and $f_2$ are said to be equivalent if   $f_1\prec f_2$ and $f_2\prec f_1$ (compare \cite[Def. 2.2.5]{BulKoz}).
\begin{rem}
Let us observe that if $\varphi_1$ and $\varphi_2$ are equivalent $C$-functions, then the $C$-functions $\exp\circ\varphi_1-1$ and $\exp\circ\varphi_2-1$ are equivalent.
Moreover, let us emphasize  that equivalent $C$-functions generate the same Orlicz space; see Krasnoselskii and Rutickii \cite[Thm. 2.8.1]{KraRu}.
\end{rem}

\begin{lem}
The functions $|\cdot|^p$ and $\varphi_p(\cdot)$ are equivalent for all $p\ge 1$.
\end{lem}
\begin{proof}
In the case of $p\ge2$, the function $|x|^p/p\le x^2/2=\varphi_p(x)$ for $|x|\le 1$ and $|x|^p/p\le |x|^p/p-1/p+1/2=\varphi_p(x)$ for $|x|\ge 1$. That is
\begin{equation}
\label{eqn1}
|x|^p/p\le \varphi_p(x)\quad {\rm for\;all}\;x\in\mathbb{R}.
\end{equation}
Let us observe now that $|x|^p/2\ge |x|^p/p$ for all $x\in\mathbb{R}$ but
\begin{equation}
\label{eqn2}
|x|^p/2\ge |x|^p/p-1/p+1/2=\varphi_p(x)\quad {\rm only\; for}\; |x|\ge 1.
\end{equation}
The inequalities (\ref{eqn1}) and (\ref{eqn2}) mean that for $p\ge 2$ the functions $|\cdot|^p$ and $\varphi_p(\cdot)$ are equivalent. In the case of $1\le p \le 2$, conversely, but similarly as above, one can show that
\begin{equation}
\label{eqn3}
|x|^p/p\ge \varphi_p(x)\quad {\rm for\;all}\;x\in\mathbb{R}.
\end{equation}
and
\begin{equation}
\label{eqn4}
|x|^p/2\le |x|^p/p-1/p+1/2=\varphi_p(x)\quad {\rm for}\; |x|\ge 1.
\end{equation}
It means that  the functions $|\cdot|^p$ and $\varphi_p(\cdot)$ are equivalent for all $p\ge 1$.


\end{proof}

It is known that the convex conjugate (the Legendre-Fenchel transform) of $|\cdot|^p/p$ is equal to $|\cdot|^q/q$, where $1/p+1/q=1$. We show that also  the convex conjugate $\varphi_p^\ast=\varphi_q$ even for $p=1$ and $q=\infty$.

\begin{lem}
\label{conjphi}
Let $p\ge 1$. Then $\varphi_p^\ast=\varphi_q$, where $1/p+1/q=1$ $(\varphi_1^\ast=\varphi_\infty)$. 
\end{lem}
\begin{proof}
Since $\varphi_p$ are even, convex and differentiable functions, it suffices to consider $x\ge 0$ and use the classical Legendre formula to get the form of $\varphi_p^\ast$ for $p>1$.
And so 
$$
\varphi_p^\prime(x)=\left\{
\begin{array}{ccl}
x, & {\rm if} & 0\le x\le 1,\\
x^{p-1}, & {\rm if} & x >1.
\end{array}
\right.
$$
Solving the equation $\varphi_p^\prime(x)=y$ we get 
$$
[\varphi_p^\prime]^{-1}(y)=\left\{
\begin{array}{ccl}
y, & {\rm if} & 0\le y\le 1,\\
y^\frac{1}{p-1}, & {\rm if} & y >1.
\end{array}
\right.
$$
Legendre's transform of $\varphi_p$ equals $\varphi_p^\ast(y)=y[\varphi_p^\prime]^{-1}(y)-\varphi_p([\varphi_p^\prime]^{-1}(y))$. 
For $0\le y\le 1$, we get $\varphi_p^\ast(y)=y^2-y^2/2=y^2/2$ and for $y >1$ we obtain
$$
\varphi_p^\ast(y)=yy^\frac{1}{p-1}-\frac{1}{p}y^\frac{p}{p-1}+\frac{1}{p}-\frac{1}{2}=\Big(1-\frac{1}{p}\Big)y^\frac{p}{p-1}+\frac{1}{p}-\frac{1}{2}= 
\frac{1}{q}y^q-\frac{1}{q}+\frac{1}{2},
$$
since $1-1/p=1/q$, $p/(p-1)=q$ and $1/p-1/2=-1/q+1/2$.

For $p=1$ and $0\le y\le 1$ we get  $\varphi_1^\ast(y)=y^2/2$ in the same manner as above. When $y>1$, by the definition of the convex conjugate, we have
$\varphi_1^\ast(y)=\sup_{x\in\mathbb{R}}\{xy-\varphi_1(x)\}=\infty$.
\end{proof}

Similarly as in Buldygin and Kozachenko \cite[(4.3) of Ch.2]{BulKoz}, let us define the norm on the spaces $L_{\psi_p}^0$ of the form
$$
\tau_{\varphi_p}(X)=\inf\big\{K>0:\;\forall_{t\in\mathbb{R}}\;\mathbb{E}\exp(tX)\le \exp\varphi_q(Kt)\big\},
$$
where $q=p/(p-1)$ if $p>1$ and $q=\infty$ for $p=1$.
\begin{thm}
\label{mgf}
The norms $\|\cdot\|_{\psi_p}$ and $\tau_{\varphi_p}(\cdot)$ are equivalent on the space $L_{\psi_p}^0$.
\end{thm}
\begin{proof}
Let $K\ge\|X\|_{\psi_p}$. Since $\mathbb{E}\xi=0$, we get
\begin{eqnarray*}
\mathbb{E}\exp(tX) &=& 1+\sum_{\alpha=2}^\infty\frac{\mathbb{E}X^\alpha}{\alpha!}t^\alpha\le 1+\sum_{\alpha=2}^\infty\frac{\mathbb{E}|X|^\alpha}{\alpha!}|t|^\alpha\\
\; &\le& 1+2\sum_{\alpha=2}^\infty\frac{\Gamma(\alpha/p+1)}{\alpha!}K^\alpha|t|^\alpha\quad({\rm by\; property\; 3\; of\; Lemma\; \ref{charlem}})\\
\; &\le& 1+2\sum_{\alpha=2}^\infty K^\alpha|t|^\alpha\quad ({\rm since}\;\Gamma(\alpha/p+1)\le\alpha!\; {\rm for}\;p\ge 1\; {\rm and}\;\alpha\ge 2 )\\
\; &\le& 1+2K^2t^2/(1-K|t|)\le 1 + 4K^2t^2
\end{eqnarray*}
for $|t|\le 1/(2K)$. Using the elementary estimate $1+x\le\exp x$, we get
$$
\mathbb{E}\exp(tX)\le \exp(4K^2t^2) 
$$
for  $|t|\le 1/(2K)$. Since $\varphi_q(x)=x^2/2$ if $|x|\le 1$, we obtain for any $q\in[1,\infty]$
\begin{equation}
\label{est1}
\mathbb{E}\exp(tX)\le \exp(4K^2t^2)= \exp\big(\varphi_q(2\sqrt{2}Kt)\big)
\end{equation}
if $|t|\le 1/(2\sqrt{2}K)$. Because $\varphi_\infty(x)=\infty$ for $|x|>1$ then, in the case $p=1$, the above inequality is valid for any $t$. Taking $K=\|X\|_{\psi_1}$, by the definition of $\tau_{\varphi_1}$, we get
\begin{equation}
\label{case1}
\tau_{\varphi_1}(X)\le 2\sqrt{2}\|X\|_{\psi_1}.
\end{equation}

From now on let $p>1$.
Similarly as above, by the Taylor expansion of the exponential function, property 3 of Lemma \ref{charlem} and elementary inequalities,  one can show that for $s\in \mathbb{R}$
\begin{eqnarray*}
\mathbb{E}\exp(|s|^p|X|^p) &=& 1+\sum_{\alpha=1}^\infty\frac{\mathbb{E}|X|^{p\alpha}}{\alpha!}|s|^{p\alpha}\le 1+2\sum_{\alpha=1}^\infty\frac{\Gamma(\alpha+1)}{\alpha!}K^{p\alpha}|s|^{p\alpha}\\
\; &=& 1+2\sum_{\alpha=1}^\infty K^{p\alpha}|s|^{p\alpha}\le \exp(4K^p|s|^p)\\
\end{eqnarray*}
for $|s|\le 1/(2^{1/p}K)$. Let us note that if $|s|\le 1/(8^{1/p}K)$ then 
\begin{equation}
\label{est3}
\mathbb{E}\exp(|s|^p|X|^p)\le \exp(1/2).
\end{equation}

Let $L>0$. By the elementary Young inequality $xy\le |x|^p/p+|y|^q/q$ ($p,q>1$) we get that for all $t\in\mathbb{R}$ the following estimate holds
$$
\mathbb{E}\exp(tX)\le \exp\big(L^q|t|^q/q\big)\mathbb{E}\exp\Big(|X|^p/(pL^p)\Big).
$$
Assuming that $L\ge (8/p)^{1/p}K$, by (\ref{est3}), we get
$$
\mathbb{E}\exp\Big(|X|^p/(pL^p)\Big)\le \exp(1/2)
$$
and, in consequance,
$$
\mathbb{E}\exp(tX)\le \exp\big(L^q|t|^q/q+1/2\big)
$$
for all $t\in \mathbb{R}$. Because for $|t|>1/(2\sqrt{2}K)$ the function 
$$
\varphi_q(2\sqrt{2}Kt)=|2\sqrt{2}Kt|^q/q-1/q+1/2>1/2
$$
then we get
\begin{eqnarray*}
\mathbb{E}\exp(tX)&\le& \exp\Big(L^q|t|^q/q+\varphi_q(2\sqrt{2}Kt)\Big)\nonumber\\
\; &= & \exp\Big(\big[L^q+(2\sqrt{2}K)^q\big]|t|^q/q-1/q+1/2\Big)\nonumber\\
\end{eqnarray*}
for $|t|>1/(2\sqrt{2}K)$.
Let us notice now that for $|t|>1/[L^q+(2\sqrt{2}K)^q\big]^{1/q}$ we get the following estimate
\begin{equation}
\label{est2}
\mathbb{E}\exp(tX)\le \exp\Big(\varphi_q\Big([L^q+(2\sqrt{2}K)^q\big]^{1/q}t\Big)\Big).
\end{equation}
Since $[L^q+(2\sqrt{2}K)^q\big]^{1/q}\ge 2\sqrt{2}K$, by (\ref{est2}) and (\ref{est1}), the inequality (\ref{est2}) is valid for all $t\in\mathbb{R}$.
Taking $K=\|X\|_{\psi_p}$ and $L=(8/p)^{1/p}\|X\|_{\psi_p}$, by the definition of the norm $\tau_{\varphi_p}$, we obtain  
$$
\tau_{\varphi_p}(X)\le [(8/p)^{q/p}+(2\sqrt{2})^q\big]^{1/q}\|X\|_{\psi_p}.
$$
By the above and (\ref{case1}) we get that the norm $\|\cdot\|_{\psi_p}$ is stronger than the norm $\tau_{\varphi_p}(\cdot)$ for any $p\ge 1$.

Now we prove the opposite relation. From now on let $K\ge \tau_{\varphi_p}(X)$.  Let us take $s>0$. Using the exponential Markov inequality and then the definition of $\tau_{\varphi_p}$ we get
\begin{eqnarray*}
\mathbb{P}(X\ge t)=\mathbb{P}\Big(e^{s X}\ge e^{s t}\Big)&\le& e^{-s t}\mathbb{E}e^{sX} \le e^{-s t}e^{\varphi_q(Ks)}\\
\; &=& e^{-\{s t-\varphi_q(Ks)\}}.
\end{eqnarray*}
Taking now infimum over $s>0$ we obtain
$$
\mathbb{P}(X\ge t)\le e^{-\sup_{s>0}\{s t-\varphi_q(Ks)\}}=e^{-\varphi_p(t/K)},
$$
since $\varphi_q$ is an even function, it follows that $\sup_{s>0}\{s t-\varphi_q(Ks)\}=\varphi^\ast_q(t/K)=\varphi_p(t/K)$.
In a similar way we obtain that $\mathbb{P}(-X\ge t)\le \exp\{-\varphi_p(t/K)\}$ and, in consequence, the following estimate
$$
\mathbb{P}(|X|\ge t)\le 2e^{-\varphi_p(t/K)}
$$
holds. 

In the case of $p\ge 2$, the function $|x|^p/p\le\varphi_p(x)$ for all $x\in\mathbb{R}$, hence 
$$
\mathbb{P}(|X|\ge t)\le 2e^{-\varphi_p(t/K)}\le 2e^{-\big(t/(p^{1/p}K)\big)^p}
$$
for all $t>0$. It means that in property 2 of Lemma \ref{charlem} one can take $L=p^{1/p}K$. Recall that for $1\le p <2$, the function $|x|^p/p\ge\varphi_p(x)$ but
$|x|^p/2\le \varphi_p(x)$ for $x\ge 1$. It gives that 
$$
\mathbb{P}(|X|\ge t)\le 2e^{-\varphi_p(t/K)}\le 2e^{-\big(t/(2^{1/p}K)\big)^p}
$$
for $t\ge K$. Observe that for $t=K$ the above right hand side equals $2e^{-1/2}$ and is greater than 1 and it follows that the above estimate is valid also for $0<t<K$ that is held for all $t$. It means that for $1\le p <2$ in property 2 of Lemma \ref{charlem} one can take $L=2^{1/p}K$. By Remark \ref{mrem} we get 
$$
\|X\|_{\psi_p}\le 3^{1/p}L_p\tau_p(X),
$$
where $L_p=p^{1/p}$ if $p\ge 2$ and $L_p=2^{1/p}$ for $1\le p <2$. It finishes the proof. 

\end{proof}
\begin{rem}
In the above theorem we have shown the equivalence relations of the norms on $L_{\psi_p}^0$-spaces for any $p\ge 1$. The case $p=1$ has a specific character, since to define the $\tau_{\varphi_1}$-norm we have used the function $\varphi_\infty$ that is not formally $C$-function. It takes infinite values. 
Although in Vershynin \cite[Prop. 2.7.1]{Ver} one can find some form of the norms equivalent relation (equivalence of conditions 4 and 5)
but without, in my opinion, the important notion of the $\tau_{\varphi_1}$-norm.

\end{rem}

In the following example we present an application of our approach to complementing classical Hoeffding's inequality.
\begin{exa}
Let $X$ be a centered random variable bounded by some positive constant $a>0$, i.e., $\mathbb{P}(|X|\le a)=1$ and $\mathbb{E}X=0$. The Hoeffding Lemma gives inequality
\begin{equation}
\label{est4}
\mathbb{E}\exp(tX)\le \exp(a^2t^2/2),
\end{equation}
which means that $\tau_{\varphi_2}(X)\le a$. This inequality leads to a tail estimate of the form
\begin{equation}
\label{Hoefest}
\mathbb{P}(|X|\ge t)\le 2\exp\big(-t^2/(2a^2)\big),
\end{equation}
which is not sharp. Let us observe that for $t>a$ we have $\mathbb{P}(|X|\ge t)=0$. We propose another form of such estimates. Recall that bounded random variables belong to $L_{\psi_p}$ for all $p\ge 1$. Since $\varphi_q(x)\ge x^2/2=\varphi_2(x)$, for $q\ge 2$, we have that 
$
\mathbb{E}\exp(tX)\le \exp(\varphi_q(at)),
$
for all $t\in \mathbb{R}$. Suppose now $1<q<2$ ($p>2$).
  Mimicking the proof of Theorem \ref{mgf}, inequality (\ref{est4}) implies that for all $q\in (1,2)$ and only $|t|\le 1/a$ 
$$
\mathbb{E}\exp(tX)\le  \exp\big(\varphi_q(at)\big);
$$
compare with (\ref{est1}). The function $\mathbb{E}\exp(|s|^p|X|^p)$ is less or equal $\exp(|s|^pa^p)$ for all $s\in \mathbb{R}$ and for $|s|\le 1/(2^{1/p}a)$ is less or equal $\exp(1/2)$. Similarly as in (\ref{est2}) we get
$$
\mathbb{E}\exp(tX)\le \exp\Big(\varphi_q\Big([(2/p)^{q/p}+1\big]^{1/q}at\Big)\Big),
$$
which means that $\tau_{\varphi_p}(X)\le [(2/p)^{q/p}+1\big]^{1/q}a$. Let us observe that for $p>2$ we have $[(2/p)^{q/p}+1\big]^{1/q}\le 2$, that is, 
$\tau_{\varphi_p(X)}\le 2a$. It gives 
$$
\mathbb{E}\exp(tX)\le  \exp\big(\varphi_q(2at)\big)\quad (1<q<2)
$$
and, in consequence, tail estimates of the form
$$
\mathbb{P}(|X|\ge t)\le 2\exp\big(-\varphi_p(t/(2a))\big)
$$
for $p>2$. Letting $p$ tend to $\infty$, we get
\begin{equation}
\label{Hoefine}
\mathbb{P}(|X|\ge t)\le 2\exp\big(-\varphi_\infty(t/(2a))\big)=
\left\{
\begin{array}{ccl}
2\exp\Big(-\frac{t^2}{8a^2}\Big) & {\rm if} & |t|\le 2a,\\
0 & {\rm if} & |t|>2a.
\end{array}
\right.
\end{equation}
This estimate for $|t|\le 2a$ has a slightly worse character than in Hoeffding's case but for $|t|>2a$ it becomes accurate.

The Hoeffding inequality for sums of independent, zero-mean random variables $(X_k)$ such that $\mathbb{P}(|X_k|\le a_k)=1$, $k=1,2,...,n$, is based on the estimate
$$
\tau_{\varphi_2}\Big(\sum_{k=1}^nX_k\Big)\le \Big(\sum_{k=1}^na_k^2\Big)^{1/2},
$$
see \cite[Cor. 1.1.3]{BulKoz}. Taking in (\ref{Hoefest})  $X=\sum_{k=1}^nX_k$ and $a=(\sum_{k=1}^na_k^2)^{1/2}$ we get the classical Hoeffding inequality. 

Let us emphasize that for any $p\ge 1$ we can only use  the triangle inequality that is true even for  dependent summands $X_i$. Thus for $p>2$ we get the inequality  
$$
\tau_{\varphi_p}\Big(\sum_{k=1}^nX_k\Big)\le \sum_{k=1}^na_k.
$$
Substituting in (\ref{Hoefine}) $X=\sum_{k=1}^nX_k$ and $a=\sum_{k=1}^na_k$ we obtain a complementary form of the Hoeffding inequality for dependent random variables.

\end{exa}


\begin{thebibliography}{                    }

\bibitem{A-AGM}
Artstein-Avidan S., Giannopoulos A., Milman V.D.: Asymptotic Geometric Analysis, Part I, vol. 202, Mathematical Surveys and Monographs. American Mathematical Society, Providence, RI (2015)

\bibitem{Azuma}
 Azuma, K.: Weighted sums of certain dependent random variables, Tokohu Mathematical Journal 19 , 357-367 (1967)

\bibitem{Bour}
Bourgain J.: On the isotropy-constant problem for $\psi_2$-bodies. In: Milman, V.D., Schechtman, G. (eds.) Geometric Aspect of Functional Analysis,  Lecture Notes in Mathematics, vol. 1807, pp. 114-121. Springer, Berlin, (2003)

\bibitem{BulKoz}
 Buldygin, V., Kozachenko, Yu.: Metric Characterization of Random Variables and Random Processes, Amer.Math.Soc., Providence, RI (2000) 




\bibitem{DHPP}
Dick J.,Hinrichs A., Pillichshammer F., Prochno J.: Tractability properties of discrepancy in Orlicz norms (2019). arXiv:1910.12571 

\bibitem{GPV}
Giannopoulos A., Paouris G., Valettas P.:
$\psi_\alpha$-estimates for marginals of log-concave probability measures. Proc. Amer. Math. Soc. 140, no. 4, pp. 1297-1304 (2012)

\bibitem{HPU}
Hinrichs A., Prochno J., Ullrich M.: The curse of dimensionality for numerical integration on general domains. Journal of Complexity, Volume 50,  pp. 25-42 (2019)


\bibitem{Hoeff}
 Hoeffding, W.: Probability for sums of bounded random variables, Journal of the American Statistical Association 58, 13-30 (1963).

\bibitem{HPTh}
H\"orrmann J., Prochno J., Th\"ale Ch.: On the isotropic constant of random polytopes with vertices on an $\ell_p$-sphere. The Journal of Geometric Analysis, vol. 28 (1) pp. 405-426 (2018)

\bibitem{Kahane}
 Kahane, J.P.: Local properties of functions in terms of random Fourier series (French), Stud. Math., 19 (no. 1), 1-25 (1960).

\bibitem{BM}
Klartag B., Milman E.: Centroid bodies and the logarithmic Laplace transform - A unified approach. J. Funct. Anal. 226(1), pp. 10-34 (2012)


\bibitem{KraRu}
Krasnoselskii, M.A., Rutickii, Ja.B.:  Convex Function and Orlicz Spaces. Translated from the first edition by Leo F. Boron. P. Noordhoff Ltd. Groningen (1961)

\bibitem{Lux}
Luxemburg, W.A.J.:  Banach function spaces, T.U. Delft (Thesis) (1955) 

\bibitem{Ver} 
Vershynin, R.: High-Dimensional Probability, Cambridge University  Press (2018) 







\end{thebibliography}
\end{document}